\documentclass{amsart}
\usepackage{amsmath}
\usepackage{amsthm}
\usepackage{amsfonts}
\usepackage{amssymb}

\usepackage{graphics}
\usepackage{epsfig}
\usepackage{color}
\usepackage{verbatim}
\begin{document}
\title[Generating monotone quantities for the heat equation]{Generating monotone quantities for the heat equation}

\author{Jonathan Bennett}
\author{Neal Bez}

\address{School of Mathematics, University of Birmingham, Edgbaston, Birmingham, B15 2TT, United Kingdom}\email{j.bennett@bham.ac.uk}
\address{Department of Mathematics, Graduate School of Science and Engineering,
Saitama University, Saitama 338-8570, Japan}\email{nealbez@mail.saitama-u.ac.jp}

\thanks{This work was supported by the European Research Council [grant number 307617] (Bennett) and JSPS Kakenhi [grant numbers 26887008 and 16H05995] (Bez)}

\theoremstyle{plain}
 \newtheorem{theorem}{Theorem}[section]
 \newtheorem{corollary}[theorem]{Corollary}

\theoremstyle{definition}
 \newtheorem{remark}[theorem]{Remark}
 \newtheorem{definition}[theorem]{Definition}
  \newtheorem{example}[theorem]{Example}

\newcommand\dive{\operatorname{div}}
\newcommand\tr{\operatorname{tr}}
\newcommand\diag{\operatorname{diag}}

\parindent = 5 pt
\parskip = 12 pt

\begin{abstract}
The purpose of this article is to expose and further develop a simple yet surprisingly far-reaching framework for generating monotone quantities for positive solutions to linear heat equations in euclidean space. This framework is intimately connected to the existence of a rich variety of algebraic closure properties of families of sub/super-solutions, and more generally solutions of systems of differential inequalities capturing log-convexity properties such as the Li--Yau gradient estimate. Various applications are discussed, including connections with the general Brascamp--Lieb inequality and the Ornstein--Uhlenbeck semigroup.
\end{abstract}
\maketitle

\begin{center} \textit{To Michael Cowling on his 65th birthday.}\end{center}

\section{Introduction}
The identification of functionals which vary monotonically as their inputs
flow according to a given evolution equation is generally considered
to be more of an art than a science. Such monotonicity results have many
consequences and, in particular, allow for a deeper understanding of a variety
of important inequalities in analysis through what is often described as the \emph{semigroup method}, or \emph{semigroup interpolation}. This method, which will be illustrated several times throughout this paper, has been extraordinarily effective across mathematics, impacting on areas such as information theory, kinetic theory, geometric analysis and harmonic analysis; see for example the articles \cite{Barthe2}, \cite{CM}, \cite{LedouxICM}, \cite{Villani} and \cite{BEsc}, along with the forthcoming discussion in Section \ref{section:context}.

The purpose of this paper is to describe and develop a simple
framework within which a rich variety of monotone quantities for positive solutions of the heat equation on $\mathbb{R}^n$ may be
``generated". As will become apparent, this framework turns out to be particularly effective within geometric aspects of harmonic analysis.\footnote{Our focus is on positive solutions, and in particular, we do not claim to generate monotone functionals containing \textit{derivatives} of solutions, such as many classical Lyapunov functionals.}

Often the monotonicity of a functional under a diffusion semigroup, such as the heat semigroup, is a consequence of an underlying convexity/concavity property.
The simplest situation is when the functional itself is convex/concave. Arguing formally,
suppose that a functional $\mathcal{F}$ (acting on functions on $\mathbb{R}^n$ say) is \textit{concave} -- that is,
$$
\mathcal{F}\left(\frac{f+g}{2}\right)\geq\frac{\mathcal{F}(f)+\mathcal{F}(g)}{2},
$$
and \textit{translation-invariant} -- that is, $$\mathcal{F}(f(\cdot-a))=\mathcal{F}(f) \quad \mbox{for all $a \in\mathbb{R}^n.$}
$$
If $u:\mathbb{R}_+\times\mathbb{R}^n\rightarrow\mathbb{R}_+$ solves the heat equation $\partial_tu=\Delta u$, that is $u=H_t*f$ where $H_t$ is the heat kernel on $\mathbb{R}^n$, then
$$
\mathcal{F}(H_t*f)=\mathcal{F}\left(\int f(\cdot-y)H_t(y)dy\right)\geq\int \mathcal{F}(f(\cdot-y))H_t(y)dy=\mathcal{F}(f);
$$
thus by the semigroup property of $H_t$ we conclude that $t\mapsto\mathcal{F}(u(t,\cdot))$ is nondecreasing. It should be noticed that $H_t$ could be replaced by any positive mass-preserving convolution semigroup here. As a result the use of such a functional in semigroup interpolation will tend to generate an inequality with a very large class of extremisers (such as H\"older's inequality for example). Typically the inequalities of interest in analysis have very special extremisers and are thus associated with functionals that are \textit{not} convex/concave. In the next section we describe an approach to generating more interesting monotone quantities; see \cite{BB} for the origins of this. This will be somewhat expository in nature, with our main contributions contained in the more sophisticated framework described in Section \ref{section:systems}. Some proofs will be collected in Section \ref{section:LiYau}, followed by further results and remarks in Sections \ref{section:remarks} and \ref{section:context}.

\section{An algebraic approach}\label{introduction}

In what follows it will be convenient, and sometimes necessary, to consider more general heat equations of the form
\begin{equation}\label{genheat}
\partial_t u=\dive(A^{-1}\nabla u)
\end{equation}
where $A$ is a (symmetric) positive-definite matrix. This will allow for different diffusion coefficients and more general anisotropic flows. The reader may wish to bear in mind the particular case where $A$ is a multiple of the identity, as much of what follows is applicable in that case.

Our approach is based on two simple observations. The first is that if $u:\mathbb{R}_+\times\mathbb{R}^n\rightarrow\mathbb{R}_+$ is a (sufficiently regular) supersolution of \eqref{genheat}, that is,
\begin{equation}\label{supersol}
\partial_t u\geq\dive(A^{-1}\nabla u),
\end{equation}
then
\begin{equation}\label{obs1}
t\mapsto\int_{\mathbb{R}^n}u(t,x)\,\mathrm{d}x
\end{equation}
is nondecreasing.\footnote{Or, more generally, the quantity $t\mapsto\int_{\mathbb{R}^n}u(t,x)w(x)\,\mathrm{d}x$ is nondecreasing for any (sufficiently regular) ``$A$-subharmonic" ($-\dive(A^{-1}\nabla w)\leq 0$) weight $w$.}
This is a straightforward consequence of the divergence theorem, since (again, if $u$ is sufficiently regular),
$$
\frac{\mathrm{d}}{\mathrm{d}t}\int_{\mathbb{R}^n}u(t,x)\,\mathrm{d}x=\int_{\mathbb{R}^n} \partial_tu(t,x)\,\mathrm{d}x\geq
\int_{\mathbb{R}^n}\dive(A^{-1}\nabla u)\,\mathrm{d}x= 0.$$
The second observation is that families of supersolutions are (algebraically) closed under many more operations than families of genuine solutions. For example, if $u_1, u_2$ are supersolutions of \eqref{genheat} then their geometric mean $\widetilde{u}:=u_1^{1/2}u_2^{1/2}$ is also a supersolution. This is a consequence of the elementary identity
\begin{eqnarray*}
\begin{aligned}
\frac{\partial_t\widetilde{u}-\dive(A^{-1}\nabla\widetilde{u})}{\widetilde{u}}&=\frac{1}{2}\left(
\frac{\partial_tu_1-\dive(A^{-1}\nabla u_1)}{u_1}+\frac{\partial_tu_2-\dive(A^{-1}\nabla u_2)}{u_2}\right)\\
&\;\;\;\;\;\;\;\;\;\;\;\;+\frac{1}{4}\left\langle A^{-1}\left(\frac{\nabla u_1}{u_1}-\frac{\nabla u_2}{u_2}\right),\left(\frac{\nabla u_1}{u_1}-\frac{\nabla u_2}{u_2}\right)\right\rangle.
\end{aligned}
\end{eqnarray*}
Notice that the monotonicity of the quantity
$$
\int_{\mathbb{R}^n}u_1^{1/2}u_2^{1/2}
$$ is an immediate consequence of these two observations, as is for example, the quantity
\begin{equation}\label{repeat}
\int_{\mathbb{R}^n} u_1^{1/4}u_2^{1/4}u_3^{1/2},
\end{equation}
by a repeated application of the closure property.
Here $u_1,u_2,u_3$ are suitably regular solutions to the heat equation. As a result, provided there are enough closure properties of this type, a wide variety of monotone quantities such as
\eqref{repeat} may be \emph{generated} algebraically.

Fortunately there are many more algebraic closure properties of families of supersolutions. Below we collect together a variety of such properties, beginning with the most elementary that also preserve genuine solutions.
In what follows the differential operators $\dive$, $\nabla$ and $\Delta$ should be understood to act in the number of variables dictated by context.
\subsection{Positive linear combinations}\label{sums} If $u_1, u_2$ are supersolutions of \eqref{genheat} and $\lambda_1,\lambda_2 > 0$ then $\widetilde{u}:=\lambda_1 u_1+\lambda_2 u_2$ is also a supersolution.
\subsection{Tensor products}\label{tensorproducts} If $\partial_tu_1\geq\dive(A_1^{-1}\nabla u_1)$ and $\partial_tu_2\geq\dive(A_2^{-1}\nabla u_2)$, then $\partial_t\widetilde{u}\geq\dive((A_1\oplus A_2)^{-1}\nabla\widetilde{u})$, where $\widetilde{u}:=u_1\otimes u_2$.
\subsection{Compositions with linear transformations}\label{compositions} Suppose that $L$ is an invertible linear transformation on $\mathbb{R}^n$. Then for the spatial composition $\widetilde{u}:=u\circ L$,
$$
\partial_t u\geq\dive(A^{-1}\nabla u)\quad\Rightarrow\quad\partial_t \widetilde{u}\geq\dive((L^*AL)^{-1}\nabla \widetilde{u}).
$$
In particular, if $L$ is orthogonal with respect to the inner product $\langle x,y\rangle_A:=\langle Ax,y\rangle$, then $u$ is a supersolution to \eqref{genheat} if and only if $\widetilde{u}$ is.
\subsection{Concave images}\label{concavecase}
Suppose that $B:\mathbb{R}_+^m\rightarrow\mathbb{R}_+$ is nondecreasing (in the sense that $\partial_jB\geq 0$ for each $j=1,\hdots,m$) and is concave (that is, the hessian $D^2B$ is negative semidefinite). If $u_1,\hdots, u_m$ are supersolutions of \eqref{genheat} then $\widetilde{u}:=B(u_1,\hdots, u_m)$ is also a supersolution of \eqref{genheat}. This is a direct consequence of the identity
\begin{equation} \label{e:concaveimages}
\partial_t \widetilde{u}-\dive(A^{-1}\nabla\widetilde{u})=\sum_{j=1}^m\partial_j B(u)(\partial_t u_j-\dive(A^{-1}\nabla u_j))-\sum_{\ell=1}^n\langle D^2B(u)X^\ell,X^\ell\rangle,
\end{equation}
where $u:=(u_1,\hdots,u_m)$ and the $j$th component of $X^\ell\in\mathbb{R}^m$ is equal to the $\ell$th component of $A^{-1/2}\nabla u_j$.
Examples include the generalised geometric mean $B(x)=x_1^{p_1}\cdots x_m^{p_m}$, where the $p_j \geq 0$ are such that $p_1+\cdots+p_m=1$, and harmonic addition $B(x)=(x_1^{-1}+ \cdots +x_m^{-1})^{-1}$. In this context the function $B$ is sometimes referred to as a \textit{Bellman function}; see \cite{V1}.
\subsection{Anisotropic geometric means} \label{anisotropiccase}
The particular example of geometric means above may be extended to a certain directional setting provided we permit the diffusion matrix $A$ to change with $j$.
For each $j=1,\ldots,m$ suppose that $L_j:\mathbb{R}^n\rightarrow\mathbb{R}^{n_j}$ is a surjective linear transformation and $p_j$ is a nonnegative exponent. Suppose further that, for each $j=1,\ldots,m$, $A_j$ is a positive definite linear transformation on $\mathbb{R}^{n_j}$ such that
\begin{equation}\label{GBL}
M:=\sum_{j=1}^mp_jL_j^*A_jL_j
\end{equation}
is invertible\footnote{By taking traces and summing this reveals $p_1n_1+\cdots +p_mn_m=n$ as a necessary condition.}, and
\begin{equation}\label{BLcondition}
L_jM^{-1}L_j^*=A_j^{-1}.
\end{equation}
Then, for $$\widetilde{u}:=\prod_{j=1}^m(u_j\circ L_j)^{p_j}$$ we have that
\begin{equation} \label{e:BLclosure}
\partial_t u_j\geq\dive(A_j^{-1} \nabla u_j)\quad \text{for all $j=1,\ldots, m$} \quad \Rightarrow \quad \partial_t\widetilde{u}\geq\dive(M^{-1}\nabla\widetilde{u}).
\end{equation}
This is a direct consequence of the identity
\begin{align*}
\frac{\partial_t \widetilde{u}-\dive(M^{-1}\nabla\widetilde{u})}{\widetilde{u}}&=\sum_{j=1}^mp_j\left(\frac{\partial_t u_j-\dive(A_j^{-1}\nabla u_j)}{u_j}\right) \\
& \qquad \qquad +\sum_{j=1}^mp_j\left\langle L_j^*A_jL_j(w_j-\overline{w}),(w_j-\overline{w})\right\rangle,
\end{align*}
which may be seen from a careful inspection of the proof of Proposition 8.9 in \cite{BCCT}.
Here $w_j\in\mathbb{R}^n$ is any vector satisfying $A_jL_jw_j=\nabla(\log u_j)\circ L_j$, such as
\begin{equation}\label{defw}
w_j:=L_j^*(L_jL_j^*)^{-1}A_j^{-1}\nabla(\log u_j)\circ L_j,\quad\mbox{ and }\quad \overline{w}:=M^{-1}\sum_{j=1}^mp_jL_j^*A_jL_jw_j.
\end{equation}
This closure property, which is of course a generalisation of Closure Property \ref{compositions}, immediately recovers the monotonicity of
$$
t\mapsto \int_{\mathbb{R}^n}\prod_{j=1}^m(u_j\circ L_j)^{p_j}
$$
established in \cite{CLL} and \cite{BCCT}.
This monotonicity generates the celebrated geometric Brascamp--Lieb inequality of Ball \cite{Ball} and Barthe \cite{Barthetrans1}, \cite{Barthetrans2} by semigroup interpolation; see \cite{CLL} and \cite{BCCT} for further details. An important special case is the hypercontractive inequality for the Ornstein--Uhlenbeck semigroup.
\begin{example}[Hypercontractivity of the Ornstein--Uhlenbeck semigroup] \label{example:hyper}
Let
$$
\mathrm{d}\gamma(y) = \frac{1}{(2\pi)^{n/2}} e^{-\frac{1}{2}|y|^2} \mathrm{d}y
$$
be the normalised gaussian measure on $\mathbb{R}^n$, and let $\mathcal{L}$ be the differential operator
\[
\mathcal{L} = \Delta - \langle x,\nabla \rangle.
\]
Associated to this operator is the Ornstein--Uhlenbeck flow
\[
e^{s\mathcal{L}} F(x) = \int_{\mathbb{R}^n} F(e^{-s}x + \sqrt{1- e^{-2s}}y) \, \mathrm{d}\gamma(y)
\]
which is well-known to satisfy the hypercontractivity inequality (see \cite{Nelson})
\begin{equation} \label{e:hyper}
\| e^{s\mathcal{L}} G \|_{L^q(\mathrm{d}\gamma)} \leq \|G\|_{L^p(\mathrm{d}\gamma)}
\end{equation}
for any $1 < p < q < \infty$ and $s > 0$ which satisfy
\begin{equation} \label{e:hypercritical}
e^{2s} = \frac{q-1}{p-1}.
\end{equation}
The hypercontractivity inequality may be reduced via duality to
\[
\int_{\mathbb{R}^{2n}} B(F_1 \circ L_1, F_2 \circ L_2) \, \mathrm{d}\gamma \leq B\left(\int_{\mathbb{R}^n} F_1 \, \mathrm{d}\gamma,  \int_{\mathbb{R}^n} F_2 \, \mathrm{d}\gamma \right)
\]
for nonnegative $F_j \in L^1(\mathrm{d}\gamma)$, with
\[
B(x_1,x_2) = x_1^{\alpha_1} x_2^{\alpha_2},
\]
where
$(\alpha_1,\alpha_2) = (\frac{1}{q'}, \frac{1}{p})$, $L_j : \mathbb{R}^{2n} \to \mathbb{R}^n$ are given in block form by
\[
L_1 = (I_n \,\, 0) \quad \text{and} \quad L_2 = (\rho I_n \,\, \sqrt{1 - \rho^2}I_n),
\]
and $\rho = e^{-s}$. When formulated in this way, the hypercontractivity inequality is a consequence of the following monotonicity property proved by Ledoux in \cite{L} (see \cite{Hu} for a similar approach).
\begin{theorem} \label{t:hypermon}
Let $F_j \in L^1(\mathrm{d}\gamma)$ be nonnegative and $U_j(t,\cdot) = e^{t\mathcal{L}}F_j$ for $j=1,2$. If
\[
\widetilde{U}(t,x) = B(U_1(t,L_1x), U_2(t,L_2x)),
\]
then
\begin{equation} \label{e:Ledouxmono}
t \mapsto \int_{\mathbb{R}^{2n}} \widetilde{U}(t,\cdot) \, \mathrm{d}\gamma
\end{equation}
is nondecreasing for $t > 0$.
\end{theorem}
If $u_j$ is given by
\[
e^{-\frac{1}{2}|x|^2}U_j(t,x) = e^{nt}u_j(\tfrac{1}{2}e^{2t},e^tx)
\]
then it is easily verified that $\partial_t u_j = \Delta u_j$, and we shall show that Theorem \ref{t:hypermon} is a consequence of the closure property \eqref{e:BLclosure} for an appropriate choice of mappings $L_j$ and with $m=3$ (the remaining heat-flow arising from the measure $\mathrm{d}\gamma$). \footnote{Moreover, our argument shows that Theorem \ref{t:hypermon} holds for suitably regular \emph{supersolutions} of the Ornstein--Uhlenbeck equation, the resulting $u_j$ being a supersolution of the heat equation.}

Writing $\tau = \frac{1}{2}e^{2t}$ and making the change of variables $y = e^{t}x$ reveals that
\begin{align*}
 \int_{\mathbb{R}^{2n}} \widetilde{U}(t,x) \, \mathrm{d}\gamma(x) & =  (2\tau)^{\frac{n}{2}(\alpha_1 + \alpha_2-1)}  \int_{\mathbb{R}^{2n}} u_1(\tau, L_1y)^{\alpha_1} u_2(\tau,L_2y)^{\alpha_2} \times \\
 & \qquad \qquad \qquad \qquad \qquad  \frac{\exp(\tfrac{1}{4\tau}(\alpha_1|L_1y|^2 + \alpha_2|L_2y|^2 - |y|^2))}{(4\pi \tau)^{\frac{n}{2}}}  \, \mathrm{d}y.
\end{align*}
Making the change of variables
$
z = (L_1y,L_2y)
$
leads to
\begin{align*}
\int_{\mathbb{R}^{2n}} \widetilde{U}(t,x) \, \mathrm{d}\gamma(x)
& =   \frac{(2\pi)^{\frac{n}{2}(1 - \alpha_1 - \alpha_2)}}{(1-\rho^2)^{\frac{n}{2}}}  \int_{\mathbb{R}^{2n}} u_1(\tau, L_1z)^{\alpha_1} u_2(\tau,L_1^\perp z)^{\alpha_2}   u_3(\tau,L_3z)^{\alpha_3} \, \mathrm{d}z,
\end{align*}
where $L_1^\perp = (0 \,\, I_n)$, $u_3(t,\cdot) = e^{t\Delta}\delta_0$, $\alpha_3 = 2 - \alpha_1 - \alpha_2$ and \[
L_3 = [(1+\tfrac{1}{q} - \tfrac{1}{p})(q-p)]^{-1/2} ( (\tfrac{p}{q'})^{1/2}\,I_n \quad  - (\tfrac{q}{p'})^{1/2} \,I_n  ).
\]
Observe that
\[
M := \alpha_1 L_1^*L_1 + \alpha_2 (L_1^\perp)^*L_1^\perp + \alpha_3 L_3^*L_3 = \frac{1}{1-\rho^2} \left( \begin{array}{ccccc}  I_n & - \rho I_n \\  - \rho I_n
& I_n \end{array}   \right)
\]
is clearly invertible with
\[
M^{-1} = \left( \begin{array}{ccccc}  I_n & \rho I_n \\  \rho I_n
& I_n \end{array}   \right).
\]
From this, straightforward calculations show that
$$
L_1M^{-1}L_1^* = L_1^\perp M^{-1}(L_1^\perp)^* = L_3M^{-1}L_3^*  = I_n.
$$
Thus, by \eqref{e:BLclosure}, $\widetilde{u}$ given by
$$
\widetilde{u}(\tau,z):=u_1(\tau, L_1z)^{\alpha_1} u_2(\tau,L_1^\perp z)^{\alpha_2}   u_3(\tau,L_3z)^{\alpha_3}
$$
is a supersolution for the heat equation $\partial_\tau u=\dive(M^{-1}\nabla u)$ on $\mathbb{R}^{2n}$, from which Theorem \ref{t:hypermon} follows.
\end{example}

\subsection{Convolution-based operations}\label{convolutions} This example is rather different from the previous ones as it involves a non-pointwise (and non-local) operation. As with Closure Property \ref{anisotropiccase} we will need to consider different diffusion matrices, although in this case they will \emph{necessarily} be multiples of the identity.
Let $n \geq 1$ and $0<p,p_1,p_2<\infty$ satisfy $\frac{1}{p_1}+\frac{1}{p_2}=1+\frac{1}{p}$, and $0\leq \sigma,\sigma_1,\sigma_2<\infty$ satisfy $p_1p_1'\sigma_1=p_2p_2'\sigma_2$ and $\sigma p=\sigma_1 p_1+\sigma_2 p_2$ (for example $\sigma_j=\frac{1}{p_j|p_j'|}$, $\sigma=\frac{1}{p|p'|}$).
Finally define $\widetilde{u}$ in terms of $u_1, u_2 : \mathbb{R}_+ \times \mathbb{R}^n \to \mathbb{R}_+$ by
$$
\widetilde{u}^{1/p}=u_1^{1/p_1}*u_2^{1/p_2},
$$
where convolution is taken with respect to the spatial variables on $\mathbb{R}^n$. It is shown in \cite{BB} that if $p_j\geq 1$ and $\partial_t u_j\geq\sigma_j\Delta u_j$ for $j=1,2$, then $\partial_t \widetilde{u}\geq \sigma\Delta \widetilde{u}$. A similar statement is also obtained for subsolutions provided $p_j\leq 1$; that is, if $\partial_t u_j\leq\sigma_j\Delta u_j$ for $j=1,2$ then $\partial_t \widetilde{u}\leq \sigma\Delta \widetilde{u}$.\footnote{Strictly speaking, for $\widetilde{u}$ to be sufficiently regular (and even defined at all) the above closure properties require additional technical ingredients -- see \cite{BB} for further discussion.} By semigroup interpolation these statements quickly lead to the sharp \emph{forward and reverse} Young's convolution inequalities (the sharp forward inequality due to Beckner \cite{Beckner} and Brascamp--Lieb \cite{BL}, and the sharp reverse inequality due to Brascamp--Lieb \cite{BL}). An interesting historical discussion of semigroup interpolation in related contexts, tracing back to the work of Stam \cite{Stam} in the late 1950's, may be found in \cite{Toscani}; see also the forthcoming Remark \ref{Toscaniremark}.

The Closure Properties \ref{sums}--\ref{convolutions} may be combined and iterated to generate a number of interesting and nontrivial monotone functionals from the point of view of geometric and harmonic analysis. For example, they quickly lead to direct heat-flow proofs of the sharp forward and reverse forms of the \emph{iterated} sharp Young's convolution inequality, and thus certain sharp forms of the Hausdorff--Young inequality; see \cite{BB}. Also noteworthy are the sharp Strichartz inequalities for the time-dependent free Schr\"odinger equation in spatial dimensions one and two; for example it is well known (see \cite{Foschi} and \cite{HZ}) that
\begin{equation}\label{sharpstrich}
\|e^{i s\Delta}f\|_{L^{4}_{s,x}(\mathbb{R} \times \mathbb{R}^2)} \leq
  2^{-1/2}\|f\|_{L^2(\mathbb{R}^2)},
\end{equation}
with equality occurring precisely on gaussians $f$.
For this we have the representation formula
$$
\|e^{is\Delta}f\|_{L^4_{s,x}(\mathbb{R}\times\mathbb{R}^2)}^4=\frac{1}{4}\int_O\int_{(\mathbb{R}^2)^2}f\otimes f(\rho x)f\otimes f(x) \,\mathrm{d}x\mathrm{d}\rho,
$$
which has its origins in \cite{HZ}; here $O$ is the group of isometries of $\mathbb{R}^4$ fixing the directions $(1,0,1,0)$ and $(0,1,0,1)$, and $\mathrm{d}\rho$ is the right-invariant Haar probability measure on $O$. If $u$ is a positive solution (or indeed supersolution) to the heat equation $\partial_t u=\Delta u$ on $\mathbb{R}^2$, and $\widetilde{u}:\mathbb{R}_+\times\mathbb{R}^4\rightarrow\mathbb{R}_+$ is given by
$$
\widetilde{u}:=\frac{1}{4}\int_O(u\otimes u)^{1/2}(\rho x)(u\otimes u)^{1/2}(x)\,\mathrm{d}\rho,
$$
then using Closure Properties \ref{tensorproducts}, \ref{compositions}, \ref{concavecase} and \ref{sums} in sequence, we find that the function $\widetilde{u}$ is also a supersolution. As a consequence the quantity
$$
t\mapsto\|e^{is\Delta}(u^{1/2})\|_{L^4_{s,x}(\mathbb{R}\times\mathbb{R}^2)}^4=\int_{(\mathbb{R}^2)^2}\widetilde{u}
$$
is nondecreasing for (suitably regular) solutions $u$ of the heat equation. Comparing the limiting behaviour of this as $t\rightarrow 0$ and $t\rightarrow\infty$ is quickly seen to recover \eqref{sharpstrich}; see \cite{BBCH} for the origins of this argument.

The main purpose of this paper is to explore how these algebraic closure properties may be adapted to handle more sophisticated \emph{systems} of differential inequalities featuring natural log-convexity estimates, such as the well-known Li--Yau gradient estimate. This will allow us to generate wider classes of monotone quantities which rely on such additional convexity ingredients. This approach is already touched upon in \cite{BB}, and implicitly in  \cite{BCCT}.

\section{Algebraic properties of systems of differential inequalities}\label{section:systems}
As we have seen through the algebraic closure properties described so far, classes of supersolutions of the heat equation \eqref{genheat} exhibit much less rigid structure than classes of genuine solutions. While focusing exclusively on supersolutions has proved to be of some benefit from the point of view of systematically generating monotone quantities, it does unfortunately require us to abandon certain subtle and potentially useful structural properties enjoyed by genuine solutions.
In particular, it is well known that a genuine solution of \eqref{genheat} satisfies the log-convexity estimate
\begin{equation}\label{liyau}
D^2(\log u)\geq -\frac{A}{2t}
\end{equation}
(in the sense that $D^2(\log u)+\frac{A}{2t}$ is positive semidefinite), a statement that is equivalent to the log-convexity of $u/H_{A,t}$, where
\begin{equation}\label{HeatKer}
H_{A,t}(x):=\det\left(\frac{A}{4\pi t}\right)e^{-\frac{\langle Ax,x\rangle}{4t}}
\end{equation}
is the heat kernel. We note that on taking the trace of this inequality we obtain the so-called Li--Yau\footnote{This inequality is naturally formulated in the setting of Riemannian manifolds with nonnegative Ricci curvature, and many of our conclusions -- being local in nature -- would appear to hold in such generality. We do not pursue this.} gradient estimate \cite{LY}
\begin{equation}\label{liyau1}
\Delta(\log u)\geq-\frac{\tr A}{2t};
\end{equation}
see \cite{CM} for further discussion of this inequality in the setting of semigroup interpolation.
As may be expected given the explicit role of the heat kernel in \eqref{liyau}, simple examples show that a supersolution need not satisfy the Li--Yau estimate \eqref{liyau1} (and thus \eqref{liyau}); for example, direct calculation reveals that
\begin{equation}\label{cleverexample}
u(t,x):=
\frac{1}{(4\pi t)^{n/2}} \big(e^{\frac{|x|^2}{4t}}+e^{\frac{|x-1|^2}{4t}}\big)^{-1},
 \end{equation}
which is merely a harmonic mean of translated isotropic heat kernels, is a supersolution of $\partial_t u=\Delta u$ (\eqref{genheat} with $A$ being the identity) for which \eqref{liyau1}, and thus \eqref{liyau}, fails for \emph{every} $(t,x)\in\mathbb{R}_+\times\mathbb{R}^n$.

However, in surprisingly many cases the estimate \eqref{liyau} (or alternatively \eqref{liyau1}) may simply be \emph{imposed} on supersolutions, while maintaining or extending the desired algebraic closure properties. This leads to algebraic closure properties of a system of differential inequalities incorporating both \eqref{supersol} and \eqref{liyau}, and are described in Subsections \ref{system:sums}--\ref{system:convolutions} below. For ease of comparison, these are structured as in Subsections \ref{sums}--\ref{convolutions} in the previous section. Of course these system closure properties generate a variety of monotone quantities that are not generated by those in Section \ref{introduction}.

As may be expected, the first three closure properties are straightforward to verify and are left to the reader.
\subsection{Positive linear combinations}\label{system:sums} Let $\lambda_1,\lambda_2 > 0$ and $\widetilde{u}:=\lambda_1 u_1+\lambda_2 u_2$. If
\begin{equation*}
\partial_t u_j \geq\dive(A^{-1}\nabla u_j) \qquad \text{and} \qquad D^2(\log u_j) \geq -\frac{A}{2t}
\end{equation*}
for $j=1,2$, then
\begin{equation*}
\partial_t \widetilde{u} \geq \dive(A^{-1}\nabla \widetilde{u}) \qquad \text{and} \qquad D^2(\log \widetilde{u}) \geq -\frac{A}{2t}.
\end{equation*}
This follows from the elementary fact that a positive linear superposition of log-convex functions is log-convex.
It should be noticed that this closure property \textit{decouples} in the sense that
$$
\partial_t u_j \geq\dive(A^{-1}\nabla u_j), \quad j=1,2 \quad \Rightarrow \quad \partial_t \widetilde{u} \geq \dive(A^{-1}\nabla \widetilde{u})
$$
and
$$
D^2(\log u_j) \geq -\frac{A}{2t}, \quad j=1,2 \quad\Rightarrow
\quad D^2(\log \widetilde{u}) \geq -\frac{A}{2t}
$$
independently. Our next two properties decouple in a similar way.
\subsection{Tensor products}\label{system:tensorproducts}
Let $\widetilde{u}:=u_1\otimes u_2$. If
\begin{equation*}
\partial_t u_j \geq\dive(A_j^{-1}\nabla u_j) \qquad \text{and} \qquad D^2(\log u_j) \geq -\frac{A_j}{2t}
\end{equation*}
for $j=1,2$, then
\begin{equation*}
\partial_t \widetilde{u} \geq \dive((A_1 \oplus A_2)^{-1}\nabla \widetilde{u}) \qquad \text{and} \qquad D^2(\log \widetilde{u}) \geq -\frac{A_1 \oplus A_2}{2t}.
\end{equation*}
\subsection{Compositions with linear transformations}\label{system:compositions}
Let $\widetilde{u}:=u\circ L$, where $L$ is an invertible linear transformation on $\mathbb{R}^n$. If
\begin{equation*}
\partial_t u \geq\dive(A^{-1}\nabla u) \qquad \text{and} \qquad D^2(\log u) \geq -\frac{A}{2t}
\end{equation*}
then
\begin{equation*}
\partial_t \widetilde{u} \geq\dive((L^*AL)^{-1}\nabla \widetilde{u}) \qquad \text{and} \qquad D^2(\log \widetilde{u}) \geq -\frac{L^*AL}{2t}.
\end{equation*}

\subsection{Concave images}\label{system:concavecase} In this context the desired closure property can fail unless we impose a further condition on the concave function $B$.
\begin{theorem} \label{t:lambdazero}
Suppose $B:\mathbb{R}_+^m\rightarrow\mathbb{R}_+$ is concave and increasing in each variable. If $u_j:\mathbb{R}_+\times\mathbb{R}^{n}\rightarrow\mathbb{R}_+$ is given such that
\[
\partial_t u_j\geq \dive(A^{-1}\nabla u_j)
\]
for all $j=1,\ldots,m$, then $\widetilde{u}:=B(u_1,\hdots,u_m)$ satisfies
\[
\partial_t \widetilde{u}\geq \dive(A^{-1}\nabla \widetilde{u}).
\]
Moreover, if the function $\Theta:\mathbb{R}^m\rightarrow\mathbb{R}$, given by
\begin{equation} \label{e:Thetadefn}
\Theta(x_1,\hdots,x_m):=\log B(e^{x_1},\hdots,e^{x_m})
\end{equation}
is convex and
\[
D^2(\log u_j)\geq -\frac{A}{2t}
\]
for all $j=1,\ldots,m$, then
$$
D^2(\log \widetilde{u})\geq -\frac{A}{2t}.
$$
\end{theorem}
In the absence of the convexity of $\Theta$, the second part of the closure property in Theorem \ref{t:lambdazero} may fail dramatically, as our next result demonstrates. For simplicity we restrict our attention to $m=2$.
\begin{theorem}\label{classicaliffthm}
If $B:\mathbb{R}^2_+\rightarrow\mathbb{R}_+$ is increasing and homogeneous of degree $1$ then the log-convexity inequality \eqref{liyau} is preserved under $B$ if and only if $\Theta$ is convex. Moreover, if the scalar function $h:=\Theta(0,\cdot)$ is strictly concave, then there exist solutions $u_1,u_2 : \mathbb{R}_+ \times \mathbb{R}^n \to \mathbb{R}_+$ for which $\widetilde{u}:=B(u_1,u_2)$ fails to satisfy \eqref{liyau} for all $(t,x)$.
\end{theorem}
For the concave homogeneous function $B(x_1,x_2)=(x_1^{-1}+x_2^{-1})^{-1}$, a routine calculation reveals that
\begin{equation*}
h''(x) = -\frac{e^{x}}{(1+e^{x})^2} <0
\end{equation*}
for all $x\in\mathbb{R}$. This explains why the particular supersolution $u$ given by \eqref{cleverexample} fails to satisfy the Li--Yau estimate \eqref{liyau} everywhere. However, if we are prepared to sacrifice the full closure property, and focus only on the immediate monotone quantities, then naturally such considerations involving $\Theta$ cease to be relevant.

The closure property in Theorem \ref{t:lambdazero} also decouples as in Closure Properties \ref{system:sums}--\ref{system:compositions}. Our next closure property, which does \textit{not} decouple, tells us that we may relax the concavity hypothesis on $B$ provided we mitigate this with some growth in $t$.
\begin{theorem} \label{t:generalconcavity}
Suppose $B:\mathbb{R}_+^m\rightarrow\mathbb{R}_+$ is twice differentiable, increasing in each variable,
and that there exists $\lambda \geq 0$ such that
\begin{equation} \label{e:generalconcavity}
D^2B(x) \leq \lambda \diag\bigg(\frac{\partial_1B(x)}{x_1},\cdots,\frac{\partial_mB(x)}{x_m}\bigg)\qquad \mbox{for all $x \in \mathbb{R}_+^m$}.
\end{equation}
If $u_j:\mathbb{R}_+\times\mathbb{R}^{n}\rightarrow\mathbb{R}_+$ is given such that
\[
\partial_t u_j  \geq  \dive(A^{-1}\nabla u_j) \quad \mbox{ and } \quad D^2(\log u_j)\geq -\frac{A}{2t}
\]
for all $j=1,\ldots,m$, then $\widetilde{u}:=t^{\frac{\lambda n}{2}}B(u_1,\hdots,u_m)$ satisfies
\[
\partial_t \widetilde{u}\geq\frac{1}{1+\lambda} \dive(A^{-1}\nabla \widetilde{u}).
\]
Moreover, if the function $\Theta:\mathbb{R}^m\rightarrow\mathbb{R}$ given by \eqref{e:Thetadefn} is convex, then
$$
\frac{1}{1+\lambda} D^2(\log \widetilde{u})\geq -\frac{A}{2t}.
$$
\end{theorem}

Of course such non-autonomous mappings $u\mapsto\widetilde{u}$ take us out of the realm of monotone \textit{functionals}
$\mathcal{F}$, as the following corollary clarifies.
\begin{corollary}\label{c:generalconcavity}
Suppose $B:\mathbb{R}_+^m\rightarrow\mathbb{R}_+$ is twice differentiable, increasing in each variable,
and that there exists $\lambda \geq 0$ such that \eqref{e:generalconcavity} holds. Then, for suitable solutions $u_1,\hdots,u_m : \mathbb{R}_+ \times \mathbb{R}^n \to \mathbb{R}_+$ of the heat equation $\partial_t u_j  =  \dive(A^{-1}\nabla u_j)$, the function
$$
t\mapsto t^{\frac{\lambda n}{2}}\int_{\mathbb{R}^n}B(u_1(t,\cdot),\ldots,u_m(t,\cdot))
$$
is nondecreasing.
\end{corollary}
\begin{example}\label{example:simplest}
The simplest interesting example of Theorem \ref{t:generalconcavity} is the case $m=1$, $B(x) = x^p$ and $\lambda=p-1$ for some
$p\geq 1$. This case follows immediately from the identities
$$
\frac{\partial_t\widetilde{u}-\dive(p^{-1}A^{-1}\nabla\widetilde{u})}{\widetilde{u}}=
p\left(\frac{\partial_t u-\dive(A^{-1}\nabla u)}{u}\right)+(p-1)\tr\left(A^{-1}\left(D^2(\log u)+\frac{A}{2t}\right)\right)
$$
and
$$
D^2(\log\widetilde{u})=pD^2(\log u),
$$
from which it is clear that the claimed closure property may not be decoupled.
In this particular case Corollary \ref{c:generalconcavity} may be found in \cite{BCCT}.
\end{example}
\begin{example}\label{example:simple}
Let $p,q \in \mathbb{R}$ be such that $p \geq \max\{1,q\}$, and let $B : \mathbb{R}^2_+ \to \mathbb{R}_+$ be given by
$$
B(x) = \|x\|_{\ell^q}^p.
$$
Then
\begin{align*}
\frac{1}{p}\bigg( D^2B(x) - &(p-1)\diag\bigg(\frac{\partial_1B(x)}{x_1},\frac{\partial_2B(x)}{x_2}\bigg)\bigg)  \\
& = (q-p)B(x)^{\frac{p-2q}{p}}(x_1x_2)^{q-2} \left(
\begin{array}{ccccc}
x_2^2 & -x_1x_2 \\
-x_1x_2 & x_1^2
\end{array}
\right)
\end{align*}
which means $B$ satisfies the hypotheses of Theorem \ref{t:generalconcavity} with $\lambda = p-1$.
Consequently, for such $p$ and $q$, and for a pair of solutions $u_1,u_2 : \mathbb{R}_+ \times \mathbb{R}^n \to \mathbb{R}_+$ of the same heat equation, the quantity
\[
t \mapsto t^{\frac{(p-1)n}{2}} \int_{\mathbb{R}^n} (u_1(t,x)^q + u_2(t,x)^q)^{\frac{p}{q}}\,\mathrm{d}x
\]
is nondecreasing.
\end{example}
Example \ref{example:simple} may also be seen by a two-fold application of the simpler Example \ref{example:simplest}, combined with Closure Property \ref{system:sums}. This alternative approach illustrates well the purpose of our algebraic framework. A further such example, combining Closure Properties \ref{system:sums}--\ref{system:concavecase}, leads to the monotonicity of the quantity
$$
t\mapsto t^{\alpha-\frac{1}{2}}\|e^{is\Delta}(u^{\alpha})\|_{L^4_{s,x}(\mathbb{R}\times\mathbb{R}^2)}
$$
for all $\alpha\geq 1/2$ and (suitably regular) solutions $u$ of the heat equation on $\mathbb{R}^2$; see Section \ref{introduction} for further discussion in the case $\alpha=1/2$.

\subsection{Anisotropic geometric means} \label{system:anisotropiccase}
\begin{theorem}\label{BLLY}
For each $j=1,\ldots,m$ suppose that $L_j:\mathbb{R}^n\rightarrow\mathbb{R}^{n_j}$ is a surjective linear transformation and $p_j$ a nonnegative exponent.
Suppose further that for each $j=1,\ldots,m$,
$A_j$ is a positive definite linear transformation on $\mathbb{R}^{n_j}$ such that
$M$, defined in \eqref{GBL},
is invertible, and
\begin{equation}\label{BLconditionineq}
L_jM^{-1}L_j^*\leq A_j^{-1}.
\end{equation}
Let $$\widetilde{u}:=t^{\frac{1}{2}\left(\sum_j p_jn_j-n\right)}\prod_{j=1}^m(u_j\circ L_j)^{p_j}.$$
If
$$
\partial_t u_j\geq\dive(A_j^{-1} \nabla u_j)\quad\mbox{ and }\quad D^2(\log u_j)\geq-\frac{A_j}{2t}
$$
for all $j=1,\ldots,m$, then
$$
\partial_t\widetilde{u}\geq\dive(M^{-1}\nabla\widetilde{u})\quad\mbox{ and }\quad D^2(\log \widetilde{u})\geq-\frac{M}{2t}.
$$
\end{theorem}
As will be apparent from its proof, Theorem \ref{BLLY} decouples if and only if \eqref{BLconditionineq} holds with equality for each $j$, and under the scaling condition $\sum_j p_jn_j=n$. This closure property has implications beyond that of the simpler Closure Property \ref{anisotropiccase}. In particular, semigroup interpolation, beginning now at time $t=1$, generates the inequality
$$
\int_{\mathbb{R}^n}\prod_{j=1}^m f_j(L_jx)^{p_j}dx\leq\frac{\prod_{j=1}^m\det(A_j)^{{p_j}/2}}{\det(M)^{1/2}}\prod_{j=1}^m
\left(\int_{\mathbb{R}^{n_j}}f_j\right)^{p_j},
$$
where for each $j$, the input function $f_j$ takes the form $H_{A_j,1}*\mu_j$, where $H_{A_j,t}$ denotes the heat kernel defined by \eqref{HeatKer}, and $\mu_j$ is a positive finite Borel measure on $\mathbb{R}^{n_j}$. As the functions $f_j$ are each solutions to heat equations at time $t=1$, this may be interpreted as a certain \textit{regularised} (or \textit{discretised}) Brascamp--Lieb inequality that permits a sharp formulation. In \cite{BCCT} such inequalities are used to give a heat-flow proof of Lieb's fundamental theorem on the exhaustion of the general Brascamp--Lieb constant by centred gaussians.

\subsection{Convolution-based operations}\label{system:convolutions}
In the context of our convolution-based operations it continues to be necessary to restrict attention to isotropic flows, where $A$ is a positive multiple of the identity. (Indeed, if the following closure properties were to hold for an anisotropic flow, then semigroup interpolation would contradict the known characterisation of extremisers for either the forward or reverse Young's convolution inequalities.)
\begin{theorem}\label{BBLY}
Define $\widetilde{u}$ in terms of $u_1, u_2$ by $\widetilde{u}^{1/p}=u_1^{1/p_1}*u_2^{1/p_2}$ where $\frac{1}{p_1}+\frac{1}{p_2}=1+\frac{1}{p}$ and $p_1,p_2\geq 1$. Suppose $\sigma_1p_2p_2'=\sigma_2p_1p_1'$ and $\sigma p=\sigma_1 p_1+\sigma_2 p_2$. Then
\begin{enumerate}
\item
If
$$
\partial_t u_j\geq\sigma_j\Delta u_j
$$
for $j=1,2$, then
$$
\partial_t \widetilde{u}\geq \sigma\Delta \widetilde{u},
$$
\item
and if
$$
D^2(\log u_j)\geq-\frac{I}{2\sigma_jt}
$$
for $j=1,2$, then
$$
D^2(\log \widetilde{u})\geq -\frac{I}{2\sigma t},
$$
\end{enumerate}
where $I$ denotes the identity on $\mathbb{R}^n$.
\end{theorem}
It is implicit in the statement of Theorem \ref{BBLY} that $u_1$ and $u_2$ satisfy certain technical hypotheses ensuring that all terms which arise are well-defined (and are satisfied when $u_1$ and $u_2$ are \emph{solutions} of the appropriate heat equation); see \cite{BB} for further discussion.

The closure property in Theorem \ref{BBLY} decouples as in \eqref{system:sums}--\eqref{system:compositions}. However, the main point here is that Theorem \ref{BBLY} may be combined with other system closure properties which may not decouple. For example, combining with Theorem
\ref{t:generalconcavity} (in the simplest case of Example \ref{example:simplest}) generates the monotonicity of expressions of the form
$t\mapsto t^\alpha\|u_1^{1/p_1}(t, \cdot)*u_2^{1/p_2}(t,\cdot)\|_{L^p(\mathbb{R}^n)}$, for $\frac{1}{p_1}+\frac{1}{p_2}\geq 1+\frac{1}{p}$ and suitable $\alpha\geq 0$; see \cite{BB} for the origins of this.

In the next section we provide proofs of Theorems \ref{classicaliffthm}, \ref{t:generalconcavity}, \ref{BLLY} and \ref{BBLY}.

\section{Proof of Theorems \ref{classicaliffthm}, \ref{t:generalconcavity}, \ref{BLLY} and \ref{BBLY}} \label{section:LiYau}

\subsection{Proof of the Theorem \ref{t:generalconcavity}}
Let $u = (u_1,\ldots,u_m)$ and perform straightforward computations to get
\begin{align*}
\partial_t \widetilde{u} - \frac{1}{1 + \lambda} \dive(A^{-1}\nabla \widetilde{u})  & = t^{\frac{\lambda n}{2}} \sum_{j=1}^m \partial_j B(u) (\partial_t u_j - \dive(A^{-1}\nabla u_j) \\
& \qquad  + \frac{\lambda}{1 + \lambda} t^{\frac{\lambda n}{2}} \sum_{j=1}^m \partial_j B(u) \dive(A^{-1}\nabla u_j) \\
& \qquad + \frac{\lambda n}{2} t^{\frac{\lambda n}{2} -1}B(u) \\
& \qquad - \frac{t^{\frac{\lambda n}{2}}}{1+ \lambda}  \sum_{j,k=1}^m \partial_{j,k}B(u)\langle A^{-1/2}\nabla u_j, A^{-1/2}\nabla u_k \rangle.
\end{align*}
Since $D^2(\log u_j) + \frac{1}{2t}A \geq 0$ it follows that the trace of $A^{-1}(D^2(\log u_j) + \frac{1}{2t}A)$ is nonnegative and therefore
\[
\dive(A^{-1}\nabla u_j) \geq \frac{|A^{-1/2}\nabla u_j|^2}{u_j} - \frac{n}{2t}u_j.
\]
Hence, applying both hypotheses on each $u_j$, and using that $B$ is increasing in each variable, we obtain
\begin{align*}
\partial_t \widetilde{u} - \frac{1}{1 + \lambda} \dive(A^{-1}\nabla \widetilde{u}) & \geq \frac{\lambda}{1 + \lambda} t^{\frac{\lambda n}{2}} \sum_{j=1}^m \partial_j B(u) \bigg(  \frac{|A^{-1/2}\nabla u_j|^2}{u_j} - \frac{n}{2t}u_j \bigg) \\
& + \frac{\lambda n}{2} t^{\frac{\lambda n}{2} -1}B(u) - \frac{t^{\frac{\lambda n}{2}}}{1+ \lambda}  \sum_{j,k=1}^m \partial_{j,k}B(u)\langle A^{-1/2}\nabla u_j, A^{-1/2}\nabla u_k \rangle
\end{align*}
and therefore
\begin{align*}
\partial_t \widetilde{u} - \frac{1}{1 + \lambda} \dive(A^{-1}\nabla \widetilde{u}) & \geq \frac{t^{\frac{\lambda n}{2}}}{1 + \lambda}  \bigg[\sum_{j=1}^m \lambda \frac{\partial_jB(u)}{u_j}|A^{-1/2}\nabla u_j|^2 \\
& \qquad - \sum_{j,k = 1}^m \partial_{j,k}B(u) \langle A^{-1/2}\nabla u_j, A^{-1/2}\nabla u_k \rangle  \bigg] \\
&  \qquad \qquad + \frac{\lambda n}{2(1+\lambda)} t^{\frac{\lambda n}{2} -1} \bigg[ (1 + \lambda)B(u) - \sum_{j=1}^m \partial_jB(u) u_j \bigg].
\end{align*}
The two terms in square brackets are nonnegative as a consequence of our hypothesis \eqref{e:generalconcavity} on the hessian of $B$. For the first square-bracketed term, this is true since
\begin{align*}
& \sum_{j=1}^m \lambda \frac{\partial_jB(u)}{u_j}|A^{-1/2}\nabla u_j|^2 - \sum_{j,k = 1}^m \partial_{j,k}B(u) \langle A^{-1/2}\nabla u_j, A^{-1/2}\nabla u_k \rangle   \\
& \qquad \qquad \qquad = \sum_{\ell = 1}^n \bigg\langle \bigg(\lambda \diag\bigg(\frac{\partial_1B(u)}{u_1},\cdots,\frac{\partial_mB(u)}{u_m}\bigg) - D^2B(u)\bigg)X^\ell,X^\ell \bigg\rangle
\end{align*}
where $X^\ell \in \mathbb{R}^m$ has $j$th component equal to the $\ell$th component of $A^{-1/2}\nabla u_j$.

For the second term, take $x \in \mathbb{R}^m_+$ and observe that
\begin{align*}
& \langle x,\nabla \rangle \bigg( (1+ \lambda)B(x) - \sum_{j=1}^m \partial_jB(x)x_j\bigg) \\
& = \lambda \langle x,\nabla B(x) \rangle - \sum_{j,k=1}^m \partial_{j,k}B(x) x_jx_k \\
 & = \bigg\langle \bigg(\lambda \diag\bigg(\frac{\partial_1B(x)}{x_1},\cdots,\frac{\partial_mB(x)}{x_m}\bigg) - D^2B(x)\bigg)x,x \bigg\rangle
\end{align*}
which is nonnegative thanks to \eqref{e:generalconcavity}. Since $B(0) \geq 0$, upon integrating, it follows that
\begin{equation} \label{e:generalconcavityconseq}
(1+ \lambda)B(x) - \sum_{j=1}^m \partial_jB(x)x_j \geq 0
\end{equation}
for all $x \in \mathbb{R}^m_+$. Hence $\partial_t \widetilde{u}\geq\frac{1}{1+\lambda} \dive(A^{-1}\nabla \widetilde{u})$.

Finally, to see the closure of the log-convexity inequality, first observe the identity
\begin{equation*}
(D^2 (\log \widetilde{u}))_{j,k} =  \sum_{i = 1}^m \partial_i \Theta(\log u) (D^2 (\log u_i))_{j,k}  + \langle D^2 \Theta(\log u) \partial_j \log u, \partial_k \log u \rangle
\end{equation*}
for the $(j,k)$ component of $D^2(\log \widetilde{u})$. From the convexity of $\Theta$ it follows that
\begin{align*}
& \sum_{j,k = 1}^{n} \langle D^2 \Theta(\log u) \partial_j \log u, \partial_k \log u \rangle X_j X_k \\
& = \bigg\langle D^2\Theta(\log u)\bigg(\sum_{j=1}^n X_j \partial_j \log u\bigg),  \sum_{k=1}^n X_k \partial_k \log u \bigg\rangle
\end{align*}
is nonnegative for each $X \in \mathbb{R}^n$, and therefore
\begin{equation*}
D^2(\log \widetilde{u}) \geq \sum_{i = 1}^m \partial_i \Theta(\log u) D^2(\log u_i).
\end{equation*}
Each $\partial_i \Theta(\log u)$ is nonnegative since $B$ is increasing in each variable, hence the log-convexity estimate imposed on each $u_i$ implies
\begin{align*}
D^2(\log \widetilde{u}) \geq - \frac{A}{2t} \sum_{i = 1}^m \partial_i \Theta(\log u).
\end{align*}
However, it is clear from the definition of $\Theta$ that
\begin{equation*}
\sum_{i=1}^m \partial_i \Theta(\log x) = \frac{1}{B(x)} \sum_{i=1}^m  \partial_i B(x)x_i
\end{equation*}
for each $x \in \mathbb{R}^m_+$, and by \eqref{e:generalconcavityconseq}, $\frac{1}{1+\lambda} D^2(\log \widetilde{u}) \geq -\frac{A}{2t}$ clearly follows.

\subsection{Proof of the Theorem \ref{classicaliffthm}}
We begin with two simple observations. The first is that since $B$ is homogeneous of degree $1$,
$$
\Theta(s_1+\tau,s_2+\tau)=\tau+\Theta(s_1,s_2)
$$
for all $\tau, s_1,s_2\in\mathbb{R}$. In particular, taking $\tau=-s_1$, we obtain $\Theta(s_1,s_2)=s_1+h(s_2-s_1)$ where we recall that $h:=\Theta(0,\cdot)$. Thus,
\begin{equation}\label{hessformula}
D^2\Theta(s_1,s_2)=h''(s_2-s_1)\left(\begin{array}{ccccc} 1 & -1 \\ -1 & 1 \end{array} \right),
\end{equation}
which is positive semidefinite for all $s$ if and only if $h''\geq 0$. The second observation, which follows directly from Euler's homogeneous function theorem and the definition of $\Theta$, is that $\partial_1\Theta,\partial_2\Theta\geq 0$ and
\begin{equation}\label{neal'}
\partial_1\Theta+\partial_2\Theta\equiv 1.
\end{equation}
From the proof of Theorem \ref{t:generalconcavity}, we have the identity
\begin{equation} \label{pos}
(D^2 (\log \widetilde{u}))_{j,k} =  \sum_{i = 1}^2 \partial_i \Theta(\log u) (D^2 (\log u_i))_{j,k}  + \langle D^2 \Theta(\log u) \partial_j \log u, \partial_k \log u \rangle
\end{equation}
and the fact that if $\Theta$ is convex then the matrix defined by the second term on the right-hand side is positive semidefinite. Thus, if $\Theta$ is convex and $u_1, u_2$ satisfy \eqref{liyau}, then by \eqref{neal'} it follows that $\widetilde{u}$ satisfies \eqref{liyau}.

Suppose now that $D^2\Theta(s)$ fails to be positive semidefinite for some $s\in\mathbb{R}^2$. By \eqref{hessformula} we have that necessarily $h''(s_2-s_1)<0$. For $a_1,a_2\in\mathbb{R}^n$ let $u_1,u_2:\mathbb{R}_+\times\mathbb{R}^n\rightarrow\mathbb{R}_+$ be given by
$$
u_j(t,x)=e^{t\Delta}\delta_{a_j}(x)=\frac{1}{(4\pi t)^{n/2}}e^{-\frac{|x-a_j|^2}{4t}}.
$$
Since $\Delta(\log u_j)=-\frac{n}{2t}$ for all $(t,x)\in\mathbb{R}_+\times\mathbb{R}^n$ and each $j=1,2$, elementary calculations using \eqref{hessformula}--\eqref{pos} reveal that
\begin{equation}\label{explicit}
\Delta(\log \widetilde{u})=-\frac{n}{2t}+h''\left(\frac{|x-a_1|^2-|x-a_2|^2}{4t}\right)\frac{|a_1-a_2|^2}{4t^2}.
\end{equation}
Finally observe that, for any given $(t,x)$, we may choose $a_1\not=a_2$ satisfying
$$
\frac{|x-a_1|^2-|x-a_2|^2}{4t}= s_2-s_1
$$
and so
$\Delta(\log\widetilde{u})<-\frac{n}{2t}$. From this it follows that if $\Theta$ is not convex, then \eqref{liyau} is not preserved. The final conclusion of Theorem \ref{classicaliffthm} is also clear from the above by choosing such $u_1,u_2$ with $a_1 \neq a_2$.

\subsection{Proof of Theorem \ref{BLLY}}
A very close inspection of the proof of Proposition 8.9 in \cite{BCCT} reveals the identity
\begin{align*}
\frac{\partial_t \widetilde{u}-\dive(M^{-1}\nabla\widetilde{u})}{\widetilde{u}}&=\sum_{j=1}^mp_j\left(\frac{\partial_t u_j-\dive(A_j^{-1}\nabla u_j)}{u_j}\right)\\
&\qquad +\sum_{j=1}^mp_j\left\langle L_j^*A_jL_j(w_j-\overline{w}),(w_j-\overline{w})\right\rangle\\
&\qquad+\sum_{j=1}^m p_j\langle (A_j^{-1}-L_jM^{-1}L_j^*)A_jL_jw_j,A_jL_jw_j\rangle\\
&\qquad+\sum_{j=1}^m p_j\tr\left((A_j^{-1}-L_jM^{-1}L_j^*)\left(D^2(\log u_j)+\frac{A_j}{2t}\right)\right),
\end{align*}
where $w_j$ and $\overline{w}$ are given in \eqref{defw}. Theorem \ref{BLLY} now follows from this identity combined with the more elementary
$$
D^2(\log\widetilde{u})=\sum_{j=1}^mp_jL_j^*D^2(\log u_j)L_j
$$
and the definition of $M$.

\subsection{Proof of Theorem \ref{BBLY}}

We begin with the closure of the log-convexity estimate. The scalar version which runs parallel to this closure property (i.e. if $\Delta(\log u_j) \geq -\frac{n}{2\sigma_j t}$ for $j=1,2$, then $\Delta(\log \widetilde{u}) \geq -\frac{n}{2\sigma t}$) was established in \cite{BB} (see Theorem 6), however since this is an independent result, we provide a proof of the matrix version stated in Theorem \ref{BBLY} (naturally, following a similar approach).

We have
\begin{equation*}
\frac{\widetilde{u}^{2/p} }{p} \partial_{ij}(\log \widetilde{u}) = \widetilde{u}^{1/p} \partial_{ij} (\widetilde{u}^{1/p}) - \partial_i (\widetilde{u}^{1/p}) \partial_j (\widetilde{u}^{1/p})
\end{equation*}
and depending on which term we apply the derivatives, we have the following alternative expressions
\begin{align*}
\partial_{ij} (\widetilde{u}^{1/p})  = \left\{\begin{array}{llll} \frac{1}{p_1p_2} u_1^{1/p_1} \partial_i(\log u_1) * u_2^{1/p_2} \partial_j(\log u_2) \\
\frac{1}{p_1^2} u_1^{1/p_1}\partial_i(\log u_1) \partial_j(\log u_1) * u_2^{1/p_2} + \frac{1}{p_1}u_1^{1/p_1} \partial_{ij}(\log u_1) * u_2^{1/p_2} \\
\frac{1}{p_2^2} u_1^{1/p_1} * u_2^{1/p_2}\partial_i(\log u_2) \partial_j(\log u_2) + \frac{1}{p_2}u_1^{1/p_1}  * u_2^{1/p_2} \partial_{ij}(\log u_2). \end{array} \right.
\end{align*}
Hence, for any $\lambda_1$ and $\lambda_2$ (to be specified momentarily) we may express
\begin{align*}
\frac{\widetilde{u}^{2/p} }{p} \partial_{ij}(\log \widetilde{u}) & = I_{ij}^1 + I_{ij}^2 + II_{ij}^1 + II_{ij}^2 + II_{ij}^3 - II_{ij}^4
\end{align*}
where
\begin{align*}
I_{ij}^1 & := \frac{\lambda_1}{p_1} \widetilde{u}^{1/p}(u_1^{1/p_1} \partial_{ij}(\log u_1) * u_2^{1/p_2}) \\
I_{ij}^2 & := \frac{\lambda_2}{p_2} \widetilde{u}^{1/p}(u_1^{1/p_1}  * u_2^{1/p_2}\partial_{ij}(\log u_2)) \\
II_{ij}^1 & := \frac{\lambda_1}{p_1^2} \widetilde{u}^{1/p}(u_1^{1/p_1}\partial_i(\log u_1) \partial_j(\log u_1) * u_2^{1/p_2}) \\
II_{ij}^2 & := \frac{\lambda_2}{p_2^2} \widetilde{u}^{1/p}(u_1^{1/p_1} * u_2^{1/p_2}\partial_i(\log u_2) \partial_j(\log u_2)) \\
II_{ij}^3 & := \frac{1 - \lambda_1 - \lambda_2}{p_1p_2}\widetilde{u}^{1/p} (u_1^{1/p_1} \partial_i(\log u_1) * u_2^{1/p_2} \partial_j(\log u_2)) \\
II_{ij}^4 & := \partial_i (u_1^{1/p_1} * u_2^{1/p_2}) \partial_j (u_1^{1/p_1} * u_2^{1/p_2}).
\end{align*}
For arbitrary $X \in \mathbb{R}^n$, from the hypothesis $D^2(\log u_k) \geq -\frac{I}{2\sigma_kt}$ for the $I^k_{ij}$ terms, we have
\begin{align*}
\frac{\widetilde{u}^{2/p} }{p} \langle D^2(\log \widetilde{u})X,X \rangle  & \geq - \frac{\widetilde{u}^{2/p}}{2t}\bigg(\frac{\lambda_1}{\sigma_1 p_1} + \frac{\lambda_2}{\sigma_2 p_2} \bigg)  |X|^2 + N
\end{align*}
where the function $N : \mathbb{R}_+ \times \mathbb{R}^n \to \mathbb{R}$ is given by
\begin{align*}
N & := \frac{\lambda_1}{p_1^2}\widetilde{u}^{1/p} (u_1^{1/p_1}\langle \nabla(\log u_1),X\rangle^2 * u_2^{1/p_2}) \\
& \qquad + \frac{\lambda_2}{p_2^2} \widetilde{u}^{1/p} (u_1^{1/p_1} * u_2^{1/p_2}\langle \nabla(\log u_2),X\rangle^2) \\
& \qquad + \frac{1 - \lambda_1 - \lambda_2}{p_1p_2}\widetilde{u}^{1/p} ( u_1^{1/p_1} \langle \nabla(\log u_1),X\rangle  * u_2^{1/p_2} \langle \nabla(\log u_2),X\rangle) \\
& \qquad - \langle \nabla(u_1^{1/p_1} * u_2^{1/p_2}),X \rangle^2.
\end{align*}
Choosing
\[
\lambda_k = \bigg(\frac{\frac{1}{p_k'}}{\frac{1}{p_1'} + \frac{1}{p_2'}}\bigg)^2
\]
it follows that
\[
\frac{\lambda_1}{\sigma_1 p_1} + \frac{\lambda_2}{\sigma_2 p_2}  = \frac{1}{\sigma p}
\]
and thus we are left to show that $N$ is a nonnegative function. However, if $g_k : \mathbb{R}_+ \times \mathbb{R}^n \to \mathbb{R}$ is given by
\[
g_k := \frac{\lambda_k^{1/2}}{p_k} \langle \nabla(\log u_k),X \rangle,
\]
then the pointwise identity
\begin{align*}
2N(x) & = \int u_1^{1/p_1}(x-y)u_2^{1/p_2}(y) u_1^{1/p_1}(x-z)u_2^{1/p_2}(z) \times \\
& \qquad \qquad \qquad  |g_1(x-y) + g_2(y) - (g_1(x-z) + g_2(z))|^2 \, \mathrm{d}y\mathrm{d}z
\end{align*}
is readily verified by expanding the square term in the integrand and making use of the fact that $\lambda_1^{1/2} + \lambda_2^{1/2} = 1$, from which it is clear that $N(x) \geq 0$.

The remaining claim in Theorem \ref{BBLY} was proved in \cite{BB} from which we may distill the identity
\begin{align*}
\frac{\partial_t \widetilde{u} - \sigma \Delta \widetilde{u}}{\widetilde{u}^{(p-2)/p}}  & = \frac{p}{p_1} \widetilde{u}^{1/p}\bigg( u_1^{1/p_1} \frac{\partial_t u_1 - \sigma_1 \Delta u_1}{u_1} * u_2^{1/p_2} \bigg) \\
&\qquad  + \frac{p}{p_2} \widetilde{u}^{1/p}\bigg( u_1^{1/p_1}  * u_2^{1/p_2} \frac{\partial_t u_2 - \sigma_2 \Delta u_2}{u_2}\bigg) \\
& \qquad + \frac{1}{2} \int u_1^{1/p_1}(x-y)u_2^{1/p_2}(y) u_1^{1/p_1}(x-z)u_2^{1/p_2}(z) \times \\
& \qquad \qquad \qquad |v_1(x-y) + v_2(y) - (v_1(x-z) + v_2(z))|^2 \, \mathrm{d}y\mathrm{d}z,
\end{align*}
where $v_k : \mathbb{R}_+ \times \mathbb{R}^n \to \mathbb{R}^n$ is given by
\[
v_k := \bigg(\frac{p\sigma_k}{p_kp_k'}\bigg)^{1/2} \nabla(\log u_k).
\]
In this way we see a striking similarity between the proofs of both closure properties in Theorem \ref{BBLY}.

\begin{remark}\label{Toscaniremark}
If we consider $n=1$ for simplicity, Toscani \cite{Toscani} observed a beautiful connection between the nonnegativity of the function $N$, defined above, and Stam's heat-flow monotonicity proof in \cite{Stam} (see also \cite{Blachman}) of the entropy power inequality
\begin{equation} \label{e:EPI}
e^{2H(Y_1+Y_2)} \geq e^{2H(Y_1)} + e^{2H(Y_2)}
\end{equation}
for independent random variables $Y_1$ and $Y_2$ (with equality for gaussian random variables). Here, $H(Y) = - \int_\mathbb{R} f \log f$ is the entropy functional of the random variable $Y$ with probability density $f$. A key inequality in the argument of Stam is
\begin{equation} \label{e:Blachman}
\alpha_1^2 I(f_1) + \alpha_2^2 I(f_2) \geq (\alpha_1 + \alpha_2)^2 I(f_1 * f_2)
\end{equation}
for probability densities $f_1, f_2$, where $I(f) = \int_{\mathbb{R}} \frac{(f')^2}{f}$ is the Fisher information of $f$, and $\alpha_1, \alpha_2$ are constants. One can check that \eqref{e:Blachman} follows from the nonnegativity of $N$, with $f_j = u^{1/p_j}$, by multiplying by $(f_1 * f_2)^{-1}$ and integrating (see \cite{Toscani} for further details on this perspective).
\end{remark}

\section{Further results}\label{section:remarks}
\subsection{Other forms of concavity}
As may be expected the Closure Properties \ref{concavecase} and \ref{anisotropiccase} have a common generalisation involving a certain directional notion of concavity introduced in this context by Ledoux \cite{L} and Ivanisvili--Volberg \cite{V1}. Here we adopt the terminology from \cite{L}.
\begin{definition}
Given a smooth function $B$ on an open subset of $\mathbb{R}^m$ and $\Gamma=(\Gamma_{j,k})$, where $\Gamma_{j,k}$ is an $n_j\times n_k$ matrix, we say that $B$ is \emph{$\Gamma$-concave} if
\begin{equation}\label{gammadef}
\sum_{j,k=1}^m\partial_{j,k}B\langle\Gamma_{j,k}v_k,v_j\rangle\leq 0
\end{equation}
for all vectors $v_j\in\mathbb{R}^{n_j}$, $j=1,\ldots,m$.
\end{definition}
If $n_j=n$ and $L_j=I$ for each $1\leq j\leq m$ then the $\Gamma$-concavity of $B$ reduces to
$$
\sum_{j,k=1}^m\partial_{j,k}B\langle v_j,v_k\rangle=\sum_{\ell=1}^n\sum_{j,k=1}^m\partial_{j,k}B v_{j,\ell}v_{k,\ell}=\sum_{\ell=1}^n\langle D^2Bv^{\ell}, v^{\ell}\rangle\leq 0
$$
for all $m$-tuples of vectors $(v_1,\hdots,v_m)\in (\mathbb{R}^n)^m$. Here $v_{k,\ell}$ denotes the $\ell$th component of $v_k$, and $v^\ell=(v_{1,\ell},\hdots,v_{m,\ell})$ for each $k=1,\ldots,m$ and $\ell = 1,\ldots,n$. This is of course equivalent to the concavity of $B$ in the classical sense.

A simple example of a related closure property is the following:
\begin{theorem}\label{gammatheorem}
Suppose $B:\mathbb{R}_+^m\rightarrow\mathbb{R}_+$ is twice differentiable and increasing in each variable. For $j=1,\ldots,m$, suppose $L_j:\mathbb{R}^n\rightarrow\mathbb{R}^{n_j}$ is an isometry (i.e. $L_jL_j^*=I$) and let $u_j:\mathbb{R}_+\times\mathbb{R}^{n_j}\rightarrow\mathbb{R}_+$ be given such that $\partial_t u_j\geq\sigma\Delta u_j$. If $B$ is $\Gamma$-concave, where $\Gamma_{j,k}=L_jL_k^*$, then
$\widetilde{u}:=B(u_1\circ L_1,\hdots,u_m\circ L_m)$ satisfies $\partial_t\widetilde{u}\geq\sigma\Delta\widetilde{u}$.
\end{theorem}

\begin{proof} Define $u\circ L=(u_1\circ L_1,\hdots, u_m\circ L_m)$ and observe that
\begin{equation*}
\Delta\widetilde{u}=\sum_{j,k=1}^m\partial_{j,k}B\langle L_jL_k^*(\nabla u_k)\circ L_k,(\nabla u_j)\circ L_j\rangle+\sum_{j=1}^m\partial_jB(u\circ L)\Delta(u_j\circ L_j).
\end{equation*}
Since $L_j$ is an isometry, $\Delta(u_j\circ L_j)=(\Delta u_j)\circ L_j$, and so
\begin{equation}\label{identity}
\partial_t\widetilde{u}-\sigma\Delta\widetilde{u}=\sum_{j=1}^m \partial_j B(u\circ L)(\partial_t u_j-\sigma\Delta u_j) -\sigma\sum_{j,k=1}^m\partial_{j,k}B\langle L_jL_k^*(\nabla u_k)\circ L_k(\nabla u_j)\circ L_j\rangle,
\end{equation}
from which the theorem follows.
\end{proof}
Similar closure properties also hold in the context of Section \ref{section:systems}.
We refer the reader to \cite{V1} and \cite{L} for further discussion of this notion of concavity and its manifestations in geometric analysis.

\subsection{Closure properties for the Ornstein--Uhlenbeck equation}
One may equally well adopt the perspective taken in this paper for supersolutions of Ornstein--Uhlenbeck equations, and in some circumstances (but not all) the theory is equivalent. Here we content ourselves with some selected observations in the isotropic case where $A = \sigma^{-1} I$ for some $\sigma > 0$, and the corresponding differential operator
\[
\mathcal{L}_\sigma = \sigma \Delta - \langle x , \nabla \rangle.
\]
For solutions $U : \mathbb{R}_+ \times \mathbb{R}^n \to \mathbb{R}_+$ of the Ornstein--Uhlenbeck equation $\partial_tU = \mathcal{L}_\sigma U$, it is appropriate to replace \eqref{liyau} with the assumption that $U$ is log-convex. We also observe that if $U$ is a (sufficiently regular) supersolution
\[
\partial_t U \geq \mathcal{L}_\sigma U
\]
then
\[
t \mapsto \int_{\mathbb{R}^n} U(t,\cdot) \, \mathrm{d}\gamma_\sigma
\]
is nondecreasing, where
$$
\mathrm{d}\gamma_\sigma(x) = \frac{1}{(2\sigma \pi)^{n/2}}e^{-|x|^2/(2\sigma)}.
$$
In the case $\sigma = 1$ we have $\mathcal{L} = \mathcal{L}_1$ and $\mathrm{d}\gamma = \mathrm{d}\gamma_1$ introduced earlier in Example \ref{example:hyper}.

The following theorem is the natural analogue of Theorem \ref{t:generalconcavity}.
\begin{theorem} \label{t:OUgeneralconcavity}
Suppose $B:\mathbb{R}_+^m\rightarrow\mathbb{R}_+$ is twice differentiable and increasing in each variable, and that there
exists $\lambda \geq 0$ such that \eqref{e:generalconcavity} holds. If $U_j: \mathbb{R}_+ \times \mathbb{R}^n \to \mathbb{R}_+$ is given such that
\[
\partial_t U_j\geq \mathcal{L}_\sigma U_j \quad \mbox{ and } \quad D^2(\log U_j)\geq 0
\]
for all $j=1,\ldots,m$, then $\widetilde{U}:=B(U_1,\hdots,U_m)$ satisfies
\[
\partial_t \widetilde{U} \geq \mathcal{L}_{\frac{\sigma}{1+\lambda}} \widetilde{U}.
\]
Moreover, if $\Theta : \mathbb{R}^m \to \mathbb{R}$ given by \eqref{e:Thetadefn} is convex, then
\[
D^2(\log \widetilde{U})\geq 0.
\]
\end{theorem}

\begin{proof}
If $U = (U_1,\ldots,U_m)$ then
\begin{align*}
\partial_t \widetilde{U} - \mathcal{L}_{\frac{\sigma}{1+\lambda}} \widetilde{U} & = \sum_{j=1}^m \partial_j B(U) (\partial_t U_j -  \mathcal{L}_\sigma U_j) \\
& \qquad + \frac{\sigma \lambda}{1 + \lambda}  \sum_{j=1}^m \partial_j B(U) \Delta U_j  - \frac{\sigma}{1+ \lambda} \sum_{j,k=1}^m \partial_{j,k}B(U)\langle \nabla U_j, \nabla U_k \rangle.
\end{align*}
Since $U_j$ is log-convex, it follows that $U_j \Delta U_j \geq |\nabla U_j|^2$, and since $B_j$ is increasing in each variable, we obtain
\begin{align*}
\frac{1 + \lambda}{\sigma} (\partial_t \widetilde{U} - \mathcal{L}_{\frac{\sigma}{1+\lambda}} \widetilde{U}) & \geq \lambda \sum_{j=1}^m  \frac{\partial_j B(U)}{U_j} |\nabla U_j|^2  -  \sum_{j,k=1}^m \partial_{j,k}B(U)\langle \nabla U_j, \nabla U_k \rangle,
\end{align*}
which is nonnegative by \eqref{e:generalconcavity}.

Also, as in the proof of Theorem \ref{t:generalconcavity}, if we assume that $\Theta$ is convex, then
\begin{equation*}
D^2(\log \widetilde{U}) \geq \sum_{i = 1}^m \partial_i \Theta(\log U) D^2(\log U_i).
\end{equation*}
Since $\Theta$ is increasing in each variable, it is clear that the log-convexity of each $U_i$ implies the log-convexity of $\widetilde{U}$.
\end{proof}
We include the above proof since Theorem \ref{t:generalconcavity} does not in general appear to imply Theorem \ref{t:OUgeneralconcavity}. If $B$ is homogeneous, then one can establish such an implication using that
\begin{equation*}
e^{-\frac{1}{2\sigma}|x|^2}(\partial_t U - \mathcal{L}_\sigma U)(t,x) = e^{(n+2)t}(\partial_t u - \sigma\Delta u)(\tfrac{1}{2}e^{2t},e^tx)
\end{equation*}
if
\begin{equation} \label{e:OUtoheat}
e^{-\frac{1}{2\sigma}|x|^2}U(t,x) = e^{nt}u(\tfrac{1}{2}e^{2t},e^tx).
\end{equation}

Applying Theorem \ref{t:OUgeneralconcavity} with $B(x) = x^p$, $\lambda=p-1\geq 0$ and $U : \mathbb{R}_+ \times \mathbb{R}^n \to \mathbb{R}_+$ satisfying $\partial_t U = \mathcal{L} U$, yields $\partial_t(U^p) \geq \mathcal{L}_{1/p}(U^p)$, and hence
\[
t \mapsto \int_{\mathbb{R}^n} U(t,\cdot)^p \, \mathrm{d}\gamma_{1/p}
\]
is nondecreasing. It is vital that the measure in the above integral depends on $p$. In fact, if we substitute with the standard gaussian measure, the quantity
\begin{equation*}
t \mapsto \int_{\mathbb{R}^n} U(t,\cdot)^p \, \mathrm{d}\gamma
\end{equation*}
is \emph{nonincreasing}\footnote{This can be seen, for example, using \eqref{e:OUtoheat} to convert to heat-flow, and then identity \eqref{e:concaveimages} with the \emph{convex} function $B(x_1,x_2) = x_1^p x_2^{1-p}$.}. This latter observation reunites us with the hypercontractivity inequality \eqref{e:hyper} and implies that the left-hand side of the inequality is nonincreasing in the parameter $s$ (and hence the well-known fact that the inequality is valid for all $s$ larger than the critical value satisfying \eqref{e:hypercritical}).

As described in Example \ref{example:hyper}, the hypercontractivity inequality may be seen as a consequence of Closure Property \ref{anisotropiccase} involving anisotropic geometric means of heat supersolutions via the monotonicity of the quantity in \eqref{e:Ledouxmono}. We observe here that Ledoux's original proof of this monotonicity may also be viewed as a closure property of Ornstein--Uhlenbeck supersolutions. For simplicity, we exhibit this for $n=1$ in which case
\begin{align*}
\partial_t \widetilde{U} - \mathcal{L} \widetilde{U} = \sum_{j=1}^2 \partial_jB(\partial_t U_j - \mathcal{L}U_j) - \sum_{j,k = 1}^2 \partial_{jk}B  \langle \Gamma_{jk}v_j,v_k\rangle
\end{align*}
where
\[
\Gamma := \left(\begin{array}{ccc}
1 & \rho \\
\rho & 1
\end{array} \right), \qquad v_j := \frac{U_j' \circ L_j}{\sqrt{2}} \left(\begin{array}{ccc}
1 \\
1
\end{array} \right)
\]
(the meaning of the notation $B$, $\rho$, $L_1$ and $L_2$ can be found in Example \ref{example:hyper}). As shown by Ledoux, $B$ is $\Gamma$-concave and hence $\partial_t U_j \geq \mathcal{L}U_j$ for $j=1,2$ implies that $\partial_t \widetilde{U} \geq \mathcal{L}\widetilde{U}$.

\section{Contextual remarks} \label{section:context}

The idea of using heat-flow monotonicity as a tool traces back to at least
the work of Linnik \cite{Linnik} and Stam \cite{Stam} in the late 1950s. Linnik gave an information-theoretic proof of the central limit theorem and Stam was focussed on a rigorous justification of the entropy power inequality \eqref{e:EPI}. The Boltzmann $H$-functional is simply a sign change away from the entropy functional, and a short time later, this connection to collisional kinetic theory was capitalised upon by McKean \cite{McKean} who built on the work of Linnik in his study of the approach to equilibrium for Kac's caricature of a Maxwellian gas. Further discussion of these early instances in the use of heat-flow monotonicity may be found in Toscani \cite{ToscaniMilan}.

Subsequent important developments include the work of Bakry--\'Emery \cite{BE} in the study of hypercontractivity and logarithmic Sobolev inequalities, and Borell \cite{BorellPotAnal}, \cite{Borell} in the context of the Pr\'ekopa--Leindler and Ehrhard inequalities from geometric analysis. Ideas of Borell were taken up by Barthe--Cordero-Erausquin \cite{BC} and Barthe--Huet \cite{BartheHuet}, resulting in a heat-flow argument which simultaneously yields the sharp forward and reverse Brascamp--Lieb inequalities (see \cite{Barthetrans1}, \cite{Barthetrans2} for further details on the reverse inequality). This approach is based on the preservation of appropriate functional inequalities under heat-flow, but appears to be different from the heat-flow proofs of the (forward) Brascamp--Lieb inequalities in \cite{CLL} and \cite{BCCT} (and hence Closure Property \ref{anisotropiccase}). Substantial further developments have also been made in Brascamp--Lieb-type inequalities outside the realm of euclidean space; see, for example, \cite{BCLM}, \cite{BCM}, \cite{CLL}.

The choice of flow is often dictated by the nature of extremisers of an underlying inequality, and whilst gaussian, heat-flow is an obvious candidate. Naturally, the semigroup technique has also been fruitful in the context of other flows. For example, Carlen--Carrillo--Loss \cite{CCL} employed a nonlinear flow based on fast-diffusion in proving a certain case of Lieb's sharp Hardy--Littlewood--Sobolev inequality \cite{LiebHLS} (see also \cite{Dolbeault}). We also mention some very recent work of Dolbeault--Esteban--Loss \cite{DEL} where it was shown how nonlinear flows of this (porous media) type on smooth compact connected Riemannian manifolds without boundaries may be used to establish the rigidity of certain elliptic PDE in such a geometric setting.

The purpose of this discussion is to provide a sense of context for the current work, rather than to provide a comprehensive historical account of major developments. The enthusiastic reader is again encouraged to look at the surveys \cite{Barthe2}, \cite{CM}, \cite{LedouxICM} and \cite{BEsc} for further examples and yet broader perspective. We continue this section with rather more focussed contextual remarks concerning the interface with harmonic analysis, and we close by identifying certain basic limitations of the perspectives of the current paper.

\subsection{Discrete flows and induction-on-scales} Despite our attempt to systematise the identification of monotone quantities for heat-flow, the viability of semigroup interpolation as a means to establish a given functional inequality remains very far from clear in most situations (a good example, which appears to be beyond the scope of this article, is the heat-flow proof of the (non-endpoint) multilinear Kakeya inequality in \cite{BCT}). However, in some circumstances at least, certain more robust discrete analogues of this semigroup method, whereby the exact monotonicity property is weakened to a suitable recursive inequality across discrete time scales, appear to be more effective. This discrete ``monotonicity", often referred to as the \textit{method of induction-on-scales}, has been very effective in harmonic analysis in recent years; see for example \cite{BHT}, \cite{BEsc2}, \cite{BBJFA}, \cite{BBFL}, \cite{BD}, \cite{BDG}, \cite{Guth2015}, where this perspective is developed in the context of a variety of useful generalisations of the Brascamp--Lieb inequality. Earlier instances in harmonic analysis, where this connection is somewhat less explicit, begin with \cite{Bou}; see also the examples \cite{Wol} and \cite{TaoEndpoint} (Proposition 3.3).

It is also appropriate to mention discrete flows of a rather different nature which have arisen in a variety of contexts. The monotonicity of entropy along the central limit theorem is manifestly concerned with a discrete evolution and this problem (often called Shannon's problem) was resolved by Artstein \emph{et al.} in \cite{ABBN}. The central limit theorem itself was used as a tool to establish Nelson's hypercontractivity inequality by Gross \cite{Gross} and similar ideas were used by Beckner \cite{Beckner2}, \cite{Beckner} in order to prove that gaussians extremise the Hausdorff--Young inequality for the Fourier transform on $\mathbb{R}^n$. An example of a discrete flow which is arguably rather more intimately connected with the semigroup method may be found in work of Carlen--Loss \cite{CL}, where a discrete-flow proof of Lieb's sharp Hardy--Littlewood--Sobolev inequality \cite{LiebHLS} was found. Here, the flow is cleverly constructed by alternatively applying the action of a certain conformal transformation with the action of spherical symmetric decreasing rearrangement, to produce a sequence of functions converging strongly to an extremiser (for any positive input). We also note that there is an intimate connection between such induction-on-scales arguments and the \textit{Bellman function technique}, which has proved to be immensely powerful in modern harmonic analysis; see \cite{NT}, for example.

\subsection{Basic limitations}
While perhaps surprisingly powerful, our algebraic framework for generating monotone quantities for positive solutions of the heat equation has some important limitations. First of all, it will only generate quantities (or functionals) that are monotone on solutions if they are also monotone on (certain) supersolutions.
This can have some unfortunate consequences. For instance, for \emph{solutions} $u$, quantities of the form
$$
\int_{\mathbb{R}^n} B(u(t,x)) \, \mathrm{d}x
$$
are invariant (modulo harmless additive constants) under subtracting linear terms from the mapping $B$, while for supersolutions they are not. For example, the quantity
\begin{align*}
\left(\int_{\mathbb{R}^n} u_1^{1/2}u_2^{1/2}\right)^2-\int_{\mathbb{R}^n} u_1\int_{\mathbb{R}^n} u_2&=\int_{\mathbb{R}^n}\int_{\mathbb{R}^n} \big((u_1\otimes u_2)^{1/2}(u_2\otimes u_1)^{1/2}- u_1\otimes u_2\big) \\
\end{align*}
is monotone on solutions, but is not necessarily monotone on supersolutions. Typically we run into such difficulties whenever a functional involves ``differences". For example, if $u_1, u_2$ are genuine solutions then
\begin{align*}
Q_p(t):=&\left(\int_{\mathbb{R}^n} u_1^{1/p}u_2^{1/p}\right)^2-\int_{\mathbb{R}^n} u_1^{2/p}\int_{\mathbb{R}^n} u_2^{2/p}\\=&\int_{\mathbb{R}^n}\int_{\mathbb{R}^n} \big((u_1\otimes u_2)^{1/p}(u_2\otimes u_1)^{1/p}- (u_1\otimes u_2)^{2/p}\big) \\
=&-\frac{1}{2}\int_{\mathbb{R}^n}\int_{\mathbb{R}^n} \big((u_1\otimes u_2)^{1/p}-(u_2\otimes u_1)^{1/p}\big)^2
\end{align*}
satisfies
\begin{align*}
Q_p'(t)&=\frac{2}{p}\left(1-\frac{1}{p}\right)\int_{\mathbb{R}^n} u_1^{1/p}u_2^{1/p}\int_{\mathbb{R}^n} |v_1-v_2|^2 u_1^{1/p}u_2^{1/p}\\
& \qquad +\frac{2}{p}\left(\frac{2}{p}-1\right)\int_{\mathbb{R}^n}\int_{\mathbb{R}^n}|u_1(x)^{1/p}u_2(y)^{1/p}v_1(x)-u_2(x)^{1/p}u_1(y)^{1/p}v_2(x)|^2\,\mathrm{d}x\mathrm{d}y,
\end{align*}
and is thus nondecreasing for any $1\leq p\leq 2$. Here $v_j=\nabla (\log u_j)$. Again, $Q_p$ is not in general nondecreasing on \emph{supersolutions} $u_1,u_2$.
\footnote{We remark that the monotonicity of $Q_p$ for any $1\leq p\leq 2$ generates the Cauchy--Schwarz inequality by semigroup interpolation.} This limitation is reflected in the first order (monotonicity) condition imposed on the function $B$ in Closure Property \ref{concavecase}, which is an unavoidable concession to our algebraic framework (of course, unlike the concavity hypothesis, this condition is not invariant under linear changes of variables involving differences).
As a result, another important non-example is the entropy functional $-\int_{\mathbb{R}^n} f\log f$, which for similar reasons is monotone on solutions, but not in general on supersolutions.

For similar reasons it also unclear how to incorporate \emph{derivatives} into such an algebraic framework. One might recall that the Dirichlet energy
$$
\int_{\mathbb{R}^n} |\nabla u|^2
$$
is decreasing under heat flow; indeed the heat equation is a gradient flow for this functional.
While it is a fact that
$$
\partial_t u=\Delta u \quad \Rightarrow \quad \partial_t\widetilde{u}\leq\Delta\widetilde{u}
$$
for $\widetilde{u}:=|\nabla u|^2$, it is of course not true that
$$
\partial_t u\leq\Delta u \quad \Rightarrow \quad \partial_t\widetilde{u}\leq\Delta\widetilde{u}.
$$
It would seem likely that a framework that would also generate such ``difference" or ``differential" quantities would be quite different from the one we present in this paper.


\begin{thebibliography}{99}

\bibitem{ABBN}
\textit{S.~Artstein}, \textit{K.~Ball}, \textit{F.~Barthe} and \textit{A.~Naor}, Solution of Shannon's problem on the monotonicity of entropy, J. Amer. Math. Soc. \textbf{17} (2004), 975--982.

\bibitem{BE}
\textit{D.~Bakry} and \textit{M.~\'Emery}, Diffusions hypercontractives, S\'eminaire de probabilit\'es, XIX, 1983/84, 177--206, Lecture Notes in Math. 1123, Springer, Berlin, 1985.

\bibitem{Ball}
\textit{K.~Ball}, Volumes of sections of cubes and related problems, Geometric Aspects of Functional Analysis (J. Lindenstrauss, V. D. Milman, eds.) Springer Lecture Notes in Math. \textbf{1376} (1989), 251--260.

\bibitem{Barthetrans1}
\textit{F.~Barthe}, In\'{e}galit\'{e}s de Brascamp--Lieb et convexit\'{e}, C. R. Acad. Sci. Paris S\'{e}r. I Math. \textbf{324} (1997), 885--888.

\bibitem{Barthetrans2}
\textit{F.~Barthe}, On a reverse form of the Brascamp--Lieb inequality, Invent. Math. \textbf{134} (1998), 355--361.

\bibitem{Barthe2}
\textit{F.~Barthe}, The Brunn--Minkowski theorem and related geometric and functional inequalities, International Congress of Mathematicians. Vol. II, 1529--1546, Eur. Math. Soc., Z\"{u}rich, 2006.

\bibitem{BC}
\textit{F.~Barthe} and \textit{D.~Cordero-Erausquin}, Inverse Brascamp-Lieb inequalities along the heat equation, Geometric aspects of functional analysis, 65--71, Lecture Notes in Math. 1850, Springer, Berlin, 2004.

\bibitem{BCLM}
\textit{F.~Barthe}, \textit{D.~Cordero-Erausquin}, \textit{M.~Ledoux} and \textit{B.~Maurey}, Correlation and Brascamp--Lieb inequalities for Markov semigroups, Int. Math. Res. Not. IMRN 2011, no. 10, 2177--2216.

\bibitem{BCM}
\textit{F.~Barthe}, \textit{D.~Cordero-Erausquin} and \textit{B.~Maurey}, Entropy of spherical marginals and related inequalities, J. Math. Pures Appl. \textbf{86} (2006), 89--99.

\bibitem{BartheHuet}
\textit{F.~Barthe} and \textit{N.~Huet}, On Gaussian Brunn-Minkowski inequalities, Studia Math. \textbf{191} (2009), 283--304.

\bibitem{Beckner2}
\textit{W.~Beckner}, Inequalities in Fourier analysis on $\mathbb{R}^n$, Proc. Nat. Acad. Sci. U.S.A. \textbf{72} (1975), 638--641.

\bibitem{Beckner}
\textit{W.~Beckner}, Inequalities in Fourier analysis, Ann. of Math. \textbf{102} (1975), 159--182.

\bibitem{BHT}
\textit{I.~Bejenaru}, \textit{S.~Herr} and \textit{D.~Tataru}, A convolution estimate for two-dimensional hypersurfaces, Rev. Mat. Iberoamericana \textbf{26} (2010), 707--728.

\bibitem{BEsc}
\textit{J.~Bennett}, Heat-flow monotonicity related to some inequalities in euclidean analysis, Harmonic analysis and partial differential equations, 85--96, Contemp. Math., 505, Amer. Math. Soc., Providence, RI, 2010.

\bibitem{BEsc2}
\textit{J.~Bennett}, Aspects of multilinear harmonic analysis related to transversality, Harmonic analysis and partial differential equations, 1--28, Contemp. Math., 612, Amer. Math. Soc., Providence, RI, 2014.

\bibitem{BB}
\textit{J.~Bennett} and \textit{N.~Bez}, Closure properties of solutions to heat inequalities, Jour. Geom. Anal. \textbf{19} (2009), 584--600.

\bibitem{BBJFA}
\textit{J.~Bennett} and \textit{N.~Bez}, Some nonlinear Brascamp--Lieb inequalities and applications to harmonic analysis, J. Funct. Anal. \textbf{259} (2010), 2520--2556.

\bibitem{BBCH}
\textit{J.~Bennett}, \textit{N.~Bez}, \textit{A.~Carbery} and \textit{D.~Hundertmark}, Heat-flow monotonicity of Strichartz norms, Anal. PDE \textbf{2} (2009), 147--158.

\bibitem{BBFL}
\textit{J.~Bennett}, \textit{N.~Bez}, \textit{T.~C.~Flock} and \textit{S.~Lee}, Stability of the Brascamp--Lieb constant and applications, to appear in Amer. J. Math., arXiv:1508.07502.

\bibitem{BCCT}
\textit{J.~Bennett}, \textit{A.~Carbery}, \textit{M.~Christ} and \textit{T.~Tao}, The Brascamp--Lieb inequalities: finiteness, structure and extremals, Geom. Funct. Anal. \textbf{17} (2007), 1343--1415.

\bibitem{BCT}
\textit{J.~Bennett}, \textit{A.~Carbery} and \textit{T.~Tao}, On the multilinear restriction and Kakeya conjectures, Acta Math. \textbf{196} (2006), 261--302.

\bibitem{Blachman}
\textit{N.~Blachman}, The convolution inequality for entropy powers, IEEE Trans. Inform. Theory \textbf{2} (1965), 267--271.

\bibitem{BorellPotAnal}
\textit{C. Borell}, Diffusion equations and geometric inequalities, Potential Anal. \textbf{12} (2000), 49--71.

\bibitem{Borell}
\textit{C.~Borell}, The Ehrhard inequality, C. R. Math. Acad. Sci. Paris  \textbf{337} (2003), 663--666.

\bibitem{Bou}
\textit{J.~Bourgain}, Besicovitch-type maximal operators and applications to Fourier analysis, Geom. Funct. Anal. \textbf{1} (1991), 147--187.

\bibitem{BD}
\textit{J.~Bourgain} and \textit{C.~Demeter}, Mean value estimates for Weyl sums in two dimensions, J. Lond. Math. Soc \textbf{94} (2016), 814--838.

\bibitem{BDG}
\textit{J.~Bourgain}, \textit{C.~Demeter}, and \textit{L.~Guth}, Proof of the main conjecture in Vinogradov's mean value theorem for degrees higher than three, Ann. of Math. \textbf{184} (2016), 633--682.

\bibitem{BL}
\textit{H.~J.~Brascamp} and \textit{E.~H.~Lieb}, Best constants in Young's inequality, its converse, and its generalization to more than three functions, Adv. Math. \textbf{20} (1976), 151--173.

\bibitem{CCL}
\textit{E.~Carlen}, \textit{J.~A.~Carrillo} and \textit{M.~Loss}, Hardy--Littlewood--Sobolev inequalities via fast diffusion flows, Proc. Nat. Acad. Sci. \textbf{107} (2010), 19696--19701.

\bibitem{CLL}
\textit{E.~A.~Carlen}, \textit{E.~H.~Lieb} and \textit{M.~Loss}, A sharp analog of Young's inequality on $S^N$ and related entropy inequalities, Jour. Geom. Anal. \textbf{14} (2004), 487--520.

\bibitem{CL}
\textit{E.~Carlen} and \textit{M.~Loss}, Extremals of functionals with competing symmetries, J. Funct. Anal. \textbf{88} (1990), 437--456.

\bibitem{CM}
\textit{T.~H.~Colding} and \textit{W.~P.~Minicozzi II}, Monotonicity and its analytic and geometric implications, Proc. Nat. Acad. Sci. \textbf{110} (2012), 19233--19236.

\bibitem{Dolbeault}
\textit{J.~Dolbeault}, Sobolev and Hardy--Littlewood--Sobolev inequalities: duality and fast diffusion, Math. Res. Lett. \textbf{18} (2011), 1037--1050.

\bibitem{DEL}
\textit{J.~Dolbeault}, \textit{M.~J.~Esteban} and \textit{M.~Loss}, Nonlinear flows and rigidity results on compact manifolds, J. Funct. Anal. \textbf{267} (2014), 1338--1363.

\bibitem{Foschi}
\textit{D.~Foschi}, Maximizers for the Strichartz inequality, J. Eur. Math. Soc. \textbf{9} (2007), 739--774.

\bibitem{Gross}
\textit{L.~Gross}, Logarithmic Sobolev inequalities, Amer. J. Math. \textbf{97} (1975), 1061--1083.

\bibitem{Guth2015}
\textit{L.~Guth}, A short proof of the multilinear Kakeya inequality, Math. Proc. Cambridge Philos. Soc.  \textbf{158} (2015), 147--153.

\bibitem{Hu} \textit{Y.~Hu}, A unified approach to several inequalities for Gaussian and diffusion measures, S\'eminaire de Probabilit\'es XXXIV. Lecture Notes in Math. 1729, pages 329--335, Springer (2000).


\bibitem{HZ}
\textit{D.~Hundertmark} and \textit{V.~Zharnitsky}, On sharp Strichartz inequalities in low dimensions, Int. Math. Res. Not. (2006), Art. ID 34080, 18 pp.


\bibitem{V1}
\textit{P.~Ivanisvili} and \textit{A.~Volberg}, Hessian of Bellman functions and uniqueness of Brascamp--Lieb inequality, J. Lond. Math. Soc. \textbf{92} (2015), 657--674.

\bibitem{L}
\textit{M.~Ledoux}, Remarks on Gaussian noise stability, Brascamp--Lieb and Slepian inequalities, Geometric aspects of functional analysis, 309--333, Lecture Notes in Math. 2116, Springer (2014).

\bibitem{LedouxICM}
\textit{M.~Ledoux}, Heat flows, geometric and functional inequalities, Proceedings of the International Congress of Mathematicians, Seoul (2014), Vol. IV, 117--135.

\bibitem{LY}
\textit{P.~Li} and \textit{S.~Yau},  On the parabolic kernel of the Schr\"odinger operator, Acta Math.  \textbf{156}  (1986), 153--201.

\bibitem{LiebHLS}
\textit{E.~H.~Lieb}, Sharp constants in the Hardy--Littlewood--Sobolev and related inequalities, Ann. of Math. \textbf{118} (1983), 349--374.

\bibitem{Linnik}
\textit{Ju. V. Linnik}, An information-theoretic proof of the central limit theorem with Lindeberg conditions, Theor. Probability Appl. \textbf{4} (1959), 288--299.

\bibitem{McKean}
\textit{H.~P.~McKean}, Speed of approach to equilibrium for Kac's caricature of a Maxwellian gas, Arch. Ration. Mech. Anal. \textbf{21} (1966), 343--367.

\bibitem{NT}
\textit{F.~Nazarov} and \textit{S.~Treil}, The hunt for a Bellman function: applications to estimates for singular integral operators and to other problems of harmonic analysis, St. Petersburg Math. J. \textbf{8} (1997), 721--824.

\bibitem{Nelson}
\textit{E.~Nelson}, The free Markov field, J. Funct. Anal. \textbf{12} (1973), 211--227.

\bibitem{Stam}
\textit{A.~J.~Stam}, Some inequalities satisfied by the quantities of information of Fisher and Shannon, Information and Control \textbf{2} (1959), 101--112.

\bibitem{TaoEndpoint}
\textit{T.~Tao}, Endpoint bilinear restriction theorems for the cone, and some sharp null form estimates, Math. Z. \textbf{238} (2001), 215--268.

\bibitem{Toscani}
\textit{G.~Toscani}, Heat equation and the sharp Young's inequality, arXiv:1204.2086.

\bibitem{ToscaniMilan}
\textit{G.~Toscani}, Heat equation and convolution inequalities, Milan J. Math. \textbf{82} (2014), 183--212.

\bibitem{Villani}
\textit{C. Villani}, Entropy Production and Convergence to Equilibrium, Entropy Methods for the Boltzmann Equation,
Lecture Notes in Mathematics, Volume 1916, 2008, pp 1--70.

\bibitem{Wol}
\textit{T.~H.~Wolff}, A sharp bilinear cone restriction estimate,
Ann. of Math. \textbf{153} (2001), 661--698.


\end{thebibliography}
\end{document}